\documentclass[12pt]{article}
\usepackage{amsmath, amsthm, amscd, amsfonts, amssymb, graphicx}
\usepackage{amsmath,amssymb,latexsym}
\usepackage{mathtools}
\usepackage[T1]{fontenc}
\usepackage{graphicx,tikz,color,url}
\usetikzlibrary{decorations.pathreplacing}
\usepackage{todonotes}
\usepackage{cite}

\newtheorem{theorem}{Theorem}[section]
\newtheorem{lemma}[theorem]{Lemma}
\newtheorem{corollary}[theorem]{Corollary}
\newtheorem{proposition}[theorem]{Proposition}
\newtheorem{definition}[theorem]{Definition}

\newtheorem{observation}[theorem]{Observation}

\usepackage[normalem]{ulem}

\DeclareMathOperator {\gp} {gp}
\DeclareMathOperator {\mono} {mp}

\let\deg\relax
\DeclareMathOperator {\deg} {deg}

\DeclareMathOperator{\cp}{\,\square\,}

\title{Monophonic position sets of Cartesian and lexicographic products of graphs}
\author{Ullas Chandran S. V.$^{a,}$\footnote{svuc.math@gmail.com}
\and
Sandi Klav\v{z}ar $^{b,c,d,}$\footnote{sandi.klavzar@fmf.uni-lj.si}
\and
Neethu P. K.$^{a,}$\footnote{p.kneethu.pk@gmail.com}
\and
James Tuite$^{e,f,}$\footnote{james.t.tuite@open.ac.uk} 
\\\\
$^{a}$\small Department of Mathematics, Mahatma Gandhi College, University of Kerala, \\ \small Thiruvananthapuram, India\\
$^{b}$\small Faculty of Mathematics and Physics, University of Ljubljana, Slovenia\\
$^{c}$ \small Institute of Mathematics, Physics and Mechanics, Ljubljana, Slovenia \\
$^{d}$ \small Faculty of Natural Sciences and Mathematics, University of Maribor, Slovenia\\
$^{e}$ \small School of Mathematics and Statistics, Open University, UK\\
$^{f}$ \small Department of Informatics and Statistics, Klaip\.{e}da University, Lithuania
}
\date{\today}

\begin{document}

\maketitle

\vspace*{-7mm}
\begin{abstract}
The general position problem in graph theory asks for the number of vertices in a largest set $S$ of vertices of a graph $G$ such that no shortest path of $G$ contains more than two vertices of $S$. The analogous monophonic position problem is obtained from the general position problem by replacing ``shortest path'' by ``induced path.'' In this paper the monophonic position number is studied on Cartesian and lexicographic products of graphs. It is proved that in Cartesian products, a monophonic position set can only be in one of three canonical forms, named layered, varied, and cliquey. The monophonic position number of an arbitrary Cartesian product is bounded from below and above. The two bounds coincide if neither of the factors has simplicial vertices. A formula for the monophonic position number of a lexicographic product is given which only contains the clique number and the structure of monophonic sets of the second factor. 
\end{abstract}

\noindent {\bf Key words:} general position set; general position number; monophonic position set; monophonic position number; Cartesian product; lexicographic product

\medskip\noindent

{\bf AMS Subj.\ Class:} 05C57; 05C69

\section{Introduction}
\label{sec:intro}
The general position problem for graphs originates in Dudeney's `Puzzle with Pawns' in his book `Amusements in Mathematics'~\cite{dudeney-1917} from 1917. This problem asked the reader to find the largest number of pawns that can be placed on a chessboard without any three pawns lying on a common straight line in the plane. This problem was generalised to the setting of graph theory independently in~\cite{ullas-2016, manuel-2018a} as follows: a set of vertices $S$ in a graph $G$ is in \emph{general position} if no shortest path of $G$ contains more than two vertices of $S$. The problem then consists of finding the largest set of vertices in general position for a given graph $G$. The structure of general position sets in graphs was described in~\cite{AnaChaChaKlaTho}.

The general position problem has been the subject of intensive research, with some 48 papers appearing on the subject since 2018. For some recent developments see~\cite{Araujo-2025, KlaSam, DiKlKrTu-2025, irsic-2024, KlaKriTuiYer-2023, Roy-2025, ChaDiSreeThoTui2024+, Thomas-2024a, powers, TianXuChao-2023, yao-2022}. A survey of the problem is given in~\cite{survey}. One case of particular interest is the general position number of product graphs. The general position problem for Cartesian products has been treated in many papers, including~\cite{KlaPatRusYero-2021, KlavzarRus-2021, KorzeVesel-2023, Kruft-2024, Tian-2021a, Tian-2021b}. General position numbers of strong and lexicographic products were investigated in~\cite{Dokyeesun-2024b}.

Several variations on the general position problem have also been considered. For example, we can replace `shortest path' in the definition of the problem by some other family of paths. In~\cite{KlaRalYer-2021} the authors restrict attention to shortest paths of bounded length, whilst~\cite{Haponenko-2024} considers the problem for the widest possible family, i.e. all paths. Another important family is the \emph{induced} or \emph{monophonic} paths. Partially inspired by the extensive literature on monophonic convexity (a recent example is~\cite{neethu-2024}), the \emph{monophonic position problem} was introduced in~\cite{Thomas-2024b}. Some extremal problems for monophonic position sets are discussed in~\cite{TuiThoCha-2023+}, smallest maximal monophonic position sets are treated in~\cite{DiKlKrTu-2025} and graph colourings in which every colour class is in monophonic position are explored in~\cite{ChaDiSreeThoTui2024+}. In the present article we explore the monophonic position problem for various graph products.

The plan of this paper is as follows. The next section contains the definitions we need, and we also call up some of the basic results needed later on. In Section~\ref{sec:Cartesian} we focus on the monophonic position number of Cartesian products. We prove that a monophonic position set of a Cartesian product must take one of three canonical forms, which we call layered, varied, and cliquey. We bound the monophonic position number of an arbitrary Cartesian product from below and above and prove that if neither of the factors has simplicial vertices, then the bounds coincide. We also give a different upper bound for Cartesian products of triangle-free graphs. In Section~\ref{sec:Lexicographic products} we turn our attention to lexicographic products and give a formula for the monophonic position number which involves the clique number and the structure of monophonic sets of the second factor.

\section{Preliminaries}\label{sec:preliminary}

We now define the terminology that will be used in this paper. By a graph $G = (V(G), E(G))$ we mean a finite, undirected, simple graph. We will write $u \sim v$ if vertices $u$ and $v$ are adjacent. The \emph{open neighbourhood} $N(u)$ of $u \in V(G)$ is $\{ v \in V(G): u \sim v\} $, whilst the \emph{closed neighbourhood} $N[v]$ is defined by $N[u] = N(u) \cup \{ u\} $. The \emph{degree} of a vertex is the number of vertices in its open neighbourhood; in particular, a vertex of degree one is a \emph{leaf}. The \emph{distance} $d(u,v)$ between two vertices $u$ and $v$ in a connected graph $G$ is the length of a shortest $u,v$-path in $G$, and any such shortest $u,v$-path is a \emph{geodesic}. A path $P$ in $G$ is \emph{induced} or \emph{monophonic} if $G$ contains no chords between non-consecutive vertices of $P$. For two distinct vertices $u,v$ of a graph $G$, the \emph{monophonic interval} $J_G[u,v]$ is the set of all vertices lying on at least one monophonic $u,v$-path.  

Recall that a set $S$ of vertices in a graph $G$ is a \emph{general position set} if no shortest path in $G$ contains more than two vertices of $S$. The number of vertices in a largest general position set of $G$ is called the \emph{general position number} of $G$ and is denoted by $\gp(G)$.  A set $M \subseteq V(G)$ is a \emph{monophonic position set} of $G$ if no three vertices of $M$ lie on a common monophonic path in $G$. The \emph{monophonic position number} or \emph{mp-number} $\mono(G)$ of $G$ is the number of vertices in a largest monophonic position set of $G$. Observe that every monophonic position set of a graph $G$ is also in general position, and hence $\mono(G) \leq \gp(G)$. Any pair of vertices is in monophonic position, so for graphs with order $n(G) \geq 2$ we have $2 \leq \mono(G) \leq n(G)$. Trivially $\mono(G) = n(G)$ holds for a connected graph $G$ if and only if $G$ is complete. It was shown in~\cite{Thomas-2024b, TuiThoCha-2023+} that for any $2 \leq a \leq b$ there exists a graph $G$ with $\mono (G) = a$ and $\gp (G) = b$. Interestingly, the largest possible number of edges in a graph with order $n$ and monophonic position number $a$ is quadratic in $n$ (with constructions close in size to the $a$-partite Tur\'{a}n graph with order $n$), whereas for fixed general position number the largest size is linear in $n$, see~\cite{TuiThoCha-2023+}. Thus the general and monophonic position problems are intrinsically different. 

We will also make use of the following terminology to simplify our arguments: a monophonic path containing three or more vertices of a set $S$ is \emph{$S$-bad} (or if $S$ is clear from the context, we will simply say \emph{bad}). Hence, if $S$ is assumed to be in monophonic position, then the appearance of a bad path constitutes a contradiction.

We will denote the subgraph of $G$ induced by a subset $S \subseteq V(G)$ by $G[S]$. A vertex is \emph{simplicial} if its neighbourhood induces a clique. We define $\sigma (G) = 1$ if $G$ contains at least one simplicial vertex and $\sigma (G) = 0$ otherwise. 

The \emph{clique number} $\omega (G)$ of $G$ is the number of vertices in a maximum clique in $G$ and the \emph{independence number} $\alpha (G)$ is the number of vertices of a maximum independent set. The path of order $\ell $ will be written as $P_{\ell}$ and the cycle of length $\ell $ as $C_{\ell }$. For any positive integer $k$, we fix $[k] = \{1,\dots,k\}$.

We will need the following result on monophonic position sets. 
\begin{lemma}
{\rm\cite[Lemma 3.1]{Thomas-2024b}}
\label{lemma:2.1}
Let $G$ be a connected graph and $M\subseteq V(G)$ be a monophonic position set. Then $G[M]$ is a disjoint union of $k$ cliques $G[M]=\bigcup_{i=1}^{k} W_{i}$. If $k\geq 2$, then for each $i\in [k]$ any two vertices of $W_i$ have a common neighbour in $G-M$.
\end{lemma}
Let $G$ and $H$ be graphs. In this paper we discuss the monophonic position number of the {\em Cartesian product} $G\cp H$ and the {\em lexicographic product} $G\circ H$. Both products have vertex set $V(G)\times V(H)$. Let $(g_1,h_1),(g_2,h_2)\in V(G)\times V(H)$. In the Cartesian product $G\cp  H$ the vertices $(g_1,h_1)$ and $(g_2,h_2)$ are adjacent if (i) $g_1 \sim g_2$ in $G$ and $h_1=h_2$ or (ii) $g_1=g_2$ and $h_1\sim h_2$ in $H$. In the lexicographic product $G\circ H$ these vertices are adjacent if (i) $g_1\sim g_2$ or (ii) $g_1=g_2$ and $h_1\sim h_2$. 

If $\ast \in \{\cp, \circ\}$, then the \emph{projection mappings} $\pi _G:G \ast H \rightarrow G$ and $\pi _H: G \ast H \rightarrow H$ are given by $\pi _G(u,v) = u$ and $\pi _H(u,v) = v$, respectively. If $S \subseteq V(G \cp H)$, then the set $\{g \in V(G) :\ (g,h) \in S \text{ for some } h \in V(H)\}$ is the \emph{projection} $\pi_G(S)$ of $S$ on $G$. The projection $\pi_H(S)$ of $S$ on $H$ is defined analogously. We adopt the following conventions. If $v_0,v_1,\dots ,v_{\ell -1},v_{\ell }$ is a path $Q$ in $H$, then for any $u\in V(G)$, by $_uQ$ we denote the path \[ (u,v_0),(u,v_1),\dots ,(u,v_{\ell -1}),(u,v_{\ell })\] 
in $G \ast H$. Similarly, if $u_0, u_1,\dots ,u_{\ell -1}, u_{\ell }$ is a path in $G$, then in $G \ast H$, the path 
$$(u_0,v),(u_1,v),\dots ,(u_{\ell -1},v),(u_{\ell },v)$$ 
is denoted by $P_v$. Furthermore, a tilde over a path will denote that it is traversed in the opposite direction; for example, if $P$ is the path $u_0,u_1,\dots ,u_{\ell }$, then $\widetilde{P}$ is the path $u_{\ell },\dots, u_1,u_0$. Finally, for $u \in V(G)$ and $v \in V(H)$ we define $^uH$ to be the subgraph of $G \ast H$ induced by $\{ u\} \times V(H)$, which we call a \emph{$H$-layer}, whilst the \emph{$G$-layer} $G^v$ is the subgraph induced by $V(G) \times \{ v\} $.

\section{Cartesian products}\label{sec:Cartesian}

In this section we investigate the monophonic position number of the Cartesian product of connected graphs. In the main result (Theorem~\ref{thm:general bounds}) we give sharp upper and lower bounds for $\mono(G\cp H)$ in general, and the exact formula for the case when the factors do not contain simplicial vertices. In Theorem~\ref{thm:bipartite} we also give a different upper bound in the case that both factors are triangle-free.   

\subsection{Structure of $\mono $-sets}

Note that if $S\subseteq V(G\cp  H)$, then due to the commutativity of the  Cartesian product, any result that holds for $\pi_{G}(S)$ is also valid for $\pi_{H}(S)$ and vice versa. We begin with the following straightforward lower bound for $\mono (G \cp H)$.

\begin{observation}\label{obs:lower bound}
For any graphs $G$ and $H$, $\mono(G \cp H) \ge \max \{ \omega (G),\omega (H)\}$.
\end{observation}

\begin{proof}
The observation follows from the fact that for any graph $X$ we have $\mono(X)\ge \omega(X)$, and the clique number of the Cartesian product is given by $\omega(G\cp H) = \max\{\omega(G), \omega(H)\}$. 
\end{proof}

In order to prove the main result of this section, Theorem~\ref{thm:general bounds}, we need a sequence of lemmas.

\begin{lemma}
\label{lem:projection is in mp}
If $S$ is a monophonic position set of $G \cp  H$, then $\pi_{H}(S)$ is a monophonic position set of $H$.
\end{lemma}

\begin{proof} 
Suppose for a contradiction that $\pi_{H} (S)$ is not in monophonic position in $H$. Then there exists a set $S^{\prime} =\{v_1,v_2,v_3\} \subseteq \pi_{H}(S)$ such that there is an induced $v_1,v_3$-path $P$ in $H$ that passes through $v_2$. Since $\{v_1,v_2,v_3\} \subseteq \pi_{H}(S)$, there exist (not necessarily distinct) vertices $u_1,u_2,u_3$ of $G$ such that $(u_1,v_1)$ $(u_2,v_2)$ and $(u_3,v_3)$ belong to $S$; we will derive a contradiction by constructing a monophonic path in $G \cp H$ from $(u_1,v_1)$ to $(u_3,v_3)$ that passes through $(u_2,v_2)$. Let $Q$ and $R$ be monophonic paths in $G$ from $u_1$ to $u_2$ and from $u_2$ to $u_3$ respectively. Also, recall that $\widetilde{Q}$ is the reverse path of $Q$, i.e. the path formed by traversing $Q$ in the opposite direction from $u_2$ to $u_1$. Without loss of generality, there are four possibilities to consider: (i) $u_1$, $u_2$ and $u_3$ are pairwise distinct, (ii) $u_1 = u_2 \not = u_3$, (iii) $u_1 = u_3 \not = u_2$, and (iv) $u_1 = u_2 = u_3$. Note that the path $_{u_2}P$ passes through $(u_2,v_2)$.

 \begin{table}[ht]
\centering
\begin{tabular}{|c| c| c|c|}
\hline
Case & Section 1 & Section 2 & Section 3\\
\hline \hline
$u_1,u_2,u_3$ distinct & $Q_{v_1}$ & $_{u_2}P$ & $R_{v_3}\phantom{\begin{bmatrix} x \\ y \end{bmatrix}} \hspace*{-10mm}$ \\
\hline
$u_1 = u_2 \not = u_3$ & $_{u_1}P$ & $R_{v_3}$ & $\phantom{\begin{bmatrix} x \\ y \end{bmatrix}} \hspace*{-10mm}$ \\
\hline 
$u_1 = u_3 \not = u_2$ & $Q_{v_1}$ & $_{u_2}P$ & $\widetilde{Q}_{v_3}\phantom{\begin{bmatrix} x \\ x \end{bmatrix}} \hspace*{-10mm}$\\
\hline
$u_1 = u_2 = u_3$ & $_{u_1}P$ & $\phantom{\begin{bmatrix} x \\ y \end{bmatrix}} \hspace*{-10mm}$ & \\
\hline
\end{tabular}
\caption{Construction of $S$-bad paths}
\label{Tab:main lemma}
\end{table}
 The desired monophonic paths in $G \cp H$ are constructed by concatenating the paths in Table~\ref{Tab:main lemma} in order. Each section of the paths is monophonic by construction and it is easily verified that there are no chords between different sections of the displayed paths. Observe that $d_H(v_1,v_3) \geq 2$, so that there is no edge in $G \cp H$ between a vertex $(u,v_1)$ and $(u,v_3)$. 
 \end{proof}

 Note that the converse of Lemma~\ref{lem:projection is in mp} is not true. For a simple counterexample, consider $K_2 \cp K_2 \cong C_4$. Of course $V(C_4)$ is not in monophonic position, but both projections onto the first and second factors have order two, and so trivially are in monophonic position in the factors. The same observation holds using any pair of edges in the Cartesian product of any pair of non-empty graphs.    
 

 \begin{lemma}
 \label{lemma:either G or H layer}
 Let $G$ and $H$ be connected graphs and $S$ be a monophonic position set of $G\cp H$. If $(u,v)\in S$, then $V(^{u}{H})\cap S=\{ (u,v)\}$ or $V(G^v)\cap S=\{ (u,v)\}$.
 \end{lemma}
 
 \begin{proof} 
 Suppose that the result is not true, i.e. that there exist $u^{\prime } \not = u$ in $G$ and $v^{\prime } \not = v$ in $H$ such that $(u,v)$, $(u^{\prime },v)$ and $(u,v^{\prime })$ all belong to $S$. Let $P$ be a monophonic $u^{\prime },u$-path in $G$ and $Q$ a monophonic $v,v^{\prime }$-path in $H$. Then the concatenation of $P_v$ and $_{u}Q$ would be a monophonic $(u^{\prime },v),(u,v^{\prime })$-path in $G \cp H$ passing through $(u,v)$, a contradiction.  
\end{proof}

\begin{lemma}
\label{projection}
If $S = \{(u_i,v_i):\ i\in [r]\}$ is a monophonic position set of $G \cp H$ for some $r \geq 2$, then one of the following holds:
\begin{itemize}
    \item[$($a$)$] $S$ lies in a single $G$-layer, or $S$ lies in a single $H$-layer,
    \item[$($b$)$] $u_1,\dots ,u_{r}$ are distinct vertices of $G$ and $v_1,\dots ,v_{r}$ are distinct vertices of $H$, and neither of these sets induce a clique, or
    \item[$($c$)$] $\pi_G(S)$ is a clique of $G$ with order at least $2$ and $v_1,\dots ,v_{r}$ are distinct vertices of $H$, or $\pi_H(S)$ is a clique of $H$ with order at least $2$ and $u_1,\dots ,u_{r}$ are distinct vertices of $G$.
\end{itemize}
\end{lemma}

\begin{proof}
  Suppose that $S$ satisfies neither of the first two statements; without loss of generality, we can assume that $u_1 = u_2$ and there is some $u_i \in \pi _G(S)$ with $u_i \not = u_1$. We will show that this implies that $\pi _G(S)$ induces a clique in $G$ and that the vertices $v_1,\dots ,v_{r}$ of $H$ are all distinct. 

  Let $P$ be a monophonic $u_1,u_i$-path in $G$ and $Q$ and $R$ be monophonic $v_1,v_2$- and $v_2,v_i$-paths in $H$. As the concatenated path $_{u_1}Q$, $P_{v_2}$, $_{u_i}R$ passes through $(u_1,v_1)$, $(u_2,v_2)$ and $(u_i,v_i)$, there must be a chord between the paths $_{u_1}Q$ and $_{u_i}R$, which is the case if and only if $u_1 \sim u_i$ in $G$ and the paths $Q$ and $R$ intersect in $H$. Therefore, $u_1\sim u_i$ in $G$ for every $i$ such that $u_1\ne u_i$, so that $\pi _G(S)$ induces a connected subgraph of $G$. It now follows from Lemmas~\ref{lemma:2.1} and~\ref{lem:projection is in mp} that $\pi _G(S)$ induces a clique in $G$. 

  Suppose now that not all vertices $v_1,\dots ,v_{r}$ are distinct, say $v_i = v_j$ (and hence $u_i \not = u_j)$. As we are assuming that $S$ does not lie in a single layer, there must be a $v_k \in \pi_H (S)$ with $v_k \not = v_i$. Let $P^{\prime }$ be a monophonic $v_i,v_k$-path in $H$. As $\pi _G(S)$ is a clique, the path formed by following the edge $(u_i,v_i) \sim (u_j,v_i)$, the path $_{u_j}P^{\prime }$ and then the edge $(u_j,v_k) \sim (u_k,v_k)$ (if $u_k = u_j$ we omit the final edge) will be monophonic unless $u_i = u_k$ and there is an edge $v_i \sim v_k$ in $H$; however, in this case we can interchange $u_i$ and $u_j$ in the argument to produce a contradiction.
\end{proof}

\begin{definition}
We will call a monophonic position set of Type (a) \emph{layered}, of Type (b) \emph{varied}, and of Type (c) \emph{cliquey}. These three types of monophonic position sets are shown schematically in Fig.~\ref{fig:3-types}. 
\end{definition}

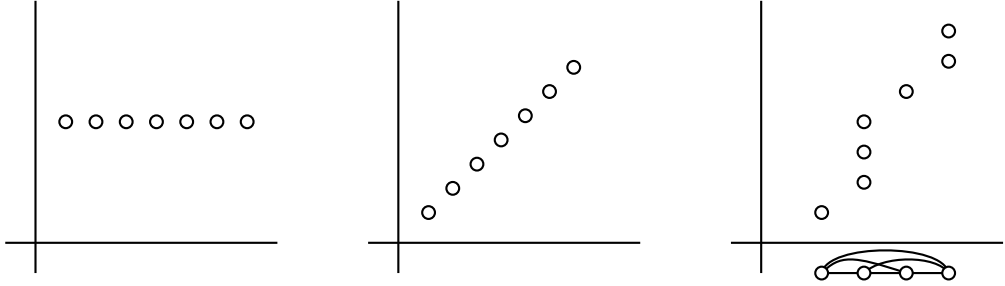
\begin{figure}[ht!]
\begin{center}
\begin{tikzpicture}[scale=0.8,style=thick,x=1cm,y=1cm]
\def\vr{3pt}

\begin{scope}[xshift=0cm, yshift=0cm] 
\coordinate(x1) at (1,2);
\coordinate(x2) at (1.5,2);
\coordinate(x3) at (2,2);
\coordinate(x4) at (2.5,2);
\coordinate(x5) at (3,2);
\coordinate(x6) at (3.5,2);
\coordinate(x7) at (4,2);
\draw (0,0) -- (4.5,0);
\draw (0.5,-0.5) -- (0.5,4);
\foreach \i in {1,2,3,4,5,6,7} 
{
\draw(x\i)[fill=white] circle(\vr);
}
\end{scope}

\begin{scope}[xshift=6cm, yshift=0cm] 
\coordinate(x1) at (1,0.5);
\coordinate(x2) at (1.4,0.9);
\coordinate(x3) at (1.8,1.3);
\coordinate(x4) at (2.2,1.7);
\coordinate(x5) at (2.6,2.1);
\coordinate(x6) at (3.0,2.5);
\coordinate(x7) at (3.4,2.9);
\draw (0,0) -- (4.5,0);
\draw (0.5,-0.5) -- (0.5,4);
\foreach \i in {1,2,3,4,5,6,7} 
{
\draw(x\i)[fill=white] circle(\vr);
}
\end{scope}

\begin{scope}[xshift=12cm, yshift=0cm] 
\coordinate(x1) at (1.5,0.5);
\coordinate(x2) at (2.2,1);
\coordinate(x3) at (2.2,1.5);
\coordinate(x4) at (2.2,2);
\coordinate(x5) at (2.9,2.5);
\coordinate(x6) at (3.6,3.0);
\coordinate(x7) at (3.6,3.5);
\coordinate(y1) at (1.5,-0.5);
\coordinate(y2) at (2.2,-0.5);
\coordinate(y3) at (2.9,-0.5);
\coordinate(y4) at (3.6,-0.5);

\draw (0,0) -- (4.5,0);
\draw (0.5,-0.5) -- (0.5,4);
\draw (y1) -- (y4);
\draw (y1) .. controls (1.8,-0.2) and (2.0,-0.2) .. (y3);
\draw (y2) .. controls (2.5,-0.2) and (3.4,-0.2) .. (y4);
\draw (y1) .. controls (1.6,0) and (3.5,0) .. (y4);
\foreach \i in {1,2,3,4,5,6,7} 
{
\draw(x\i)[fill=white] circle(\vr);
}
\foreach \i in {1,2,3,4} 
{
\draw(y\i)[fill=white] circle(\vr);
}
\end{scope}

\end{tikzpicture}
\caption{Layered (left), varied (middle), and cliquey (right) monophonic position sets}
\label{fig:3-types}
\end{center}
\end{figure}

\begin{corollary}\label{cor:upper bound}
If $G$ and $H$ are connected graphs, then 
$$\mono (G \cp H) \leq \max \{ \mono (G), \mono (H)\}\,.$$
\end{corollary}

\begin{proof}
By Lemma~\ref{projection}, $|\pi _G(S)| = |S|$ or $|\pi _H(S)| = |S|$ (or both), and so the conclusion follows from Lemma~\ref{lem:projection is in mp}.
\end{proof}

Corollary~\ref{cor:upper bound} implies that the difference $\gp (G \cp H)- \mono (G \cp H)$ can be arbitrarily large. The \emph{gear} $H_r$ is the graph of order $2r+1$ obtained by inserting an extra vertex between each pair of adjacent vertices of the outer cycle of a wheel graph $W_r$, cf.~\cite{peabody-2013}. It is shown in~\cite{Thomas-2024b} that for $r \geq 4$ the gear satisfies $\mono (H_r) = 2$ and $\gp (H_r) = r$. By Corollary~\ref{cor:upper bound} it follows that for $r \geq 4$ we have $\mono (H_r \cp H_r) = 2$, but $\gp (H_r \cp H_r) \geq \Delta (H_r \cp H_r) = 2r$.

As the bounds in Observation~\ref{obs:lower bound} and Corollary~\ref{cor:upper bound} coincide for products of paths and cycles, the monophonic position numbers of grid, cylinder and torus graphs follow immediately.

\begin{corollary}
For paths $P_m,P_n$ of order at least two and cycles $C_r,C_s$ of length at least four, 
\[ \mono (P_m \cp P_n) = \mono (P_n \cp C_r) = \mono (C_r \cp C_s) = 2\,.\]
\end{corollary}

Our bounds also make it easy to evaluate $\mono (K_n \cp H)$ for graphs $H$ with small $\mono $-number.

\begin{corollary}\label{cor:cliques,paths,cycles}
If $H$ is a connected graph and $n \geq \mono (H)$, then $\mono (K_n \cp H) = n$. 
\end{corollary}

Corollary~\ref{cor:cliques,paths,cycles} yields $\mono (K_n \cp P_m) = n$ for $n \geq 2,m \geq 4$; $\mono (K_n \cp C_m) = n$ for $n \geq 3$; and $\mono (K_n \cp K_m) = \max\{ n,m\} $.

\subsection{Varied sets}

In this subsection we show that varied monophonic position sets can only contain at most two vertices. This also incidentally allows us to determined the smallest maximal monophonic position sets in any Cartesian product.

\begin{lemma}\label{lemma:interval} 
If $u, u^{\prime} \in V(G)$ and $v, v^{\prime } \in V(H)$ are such that $u^{\prime} \not \in N_G[u]$ and $v^{\prime} \not \in N_H[v]$, then the set $\{(u,v),(u^{\prime},v^{\prime})\}$ is a maximal monophonic position set of $G\cp H$.
\end{lemma}
\begin{proof} 
Suppose that $(u,v)$ and $(u^{\prime},v^{\prime})$ satisfy the stated conditions, but that there exists a third vertex $(x,y)$ of $G \cp H$ such that $\{(u,v),(u^{\prime},v^{\prime}),(x,y)\}$ is a monophonic position set. It follows from Lemmas~\ref{lem:projection is in mp} and~\ref{projection} that $\{ u,u^{\prime },x\} $ is a set of distinct vertices of $G$ in monophonic position, and likewise $\{ v,v^{\prime },y\} $ in a monophonic position set of $H$ of order three.

Let $P$ and $P^{\prime }$ be monophonic $u,x$- and $x,u^{\prime }$-paths in $G$ respectively and $Q$ and $Q^{\prime }$ be monophonic $v,y$- and $y,v^{\prime }$-paths in $H$ respectively. Consider the path formed by the concatenation of $P_v$, $_{x}Q$, $P^{\prime }_y$ and $_{u^{\prime }}Q^{\prime }$. Trivially there are no chords between consecutive sections of the concatenated path, e.g. between $P_v$ and $_{x}Q$. There is no chord between $P_v$ and $_{u^{\prime }}Q^{\prime}$; otherwise there is an edge $(w,v) \sim (u^{\prime },z)$ in $G \cp H$, where $w$ lies on $P$ and $z$ lies on $Q^{\prime}$ and, since either $w = u^{\prime}$ or $z = v$ we would have a violation of Lemma~\ref{lem:projection is in mp}. 

Finally, if we assume that $x \not \sim u^{\prime }$ and $v \not \sim y$, then there are no chords between $P_v$ and $P^{\prime}_y$ or between $_{x}Q$ or $_{u^{\prime }}Q^{\prime}$. This contradicts our assumption that $\{(u,v),(u^{\prime},v^{\prime}),(x,y)\}$ is a monophonic position set, so we conclude that either $x \sim u^{\prime }$ or $y \sim v$. By symmetry, we also have $x \sim u$ or $y \sim v^{\prime }$. By Lemma~\ref{lem:projection is in mp} we cannot have both $x \sim u$ and $x \sim u^{\prime }$, or both $y \sim v$ and $y \sim v^{\prime }$, so without loss of generality $x \sim u$ and $y \sim v$. Let $P^{\prime \prime }$ be a monophonic $u,u^{\prime }$-path in $G$ and $Q^{\prime \prime }$ be a monophonic $v,v^{\prime }$-path in $H$. Then as $x$ does not lie on $P^{\prime \prime }$ and $y$ does not lie on $Q^{\prime \prime }$, the path formed by the edges $(x,y) \sim (x,v)$ and $(x,v) \sim (u,v)$ followed by the path $_{u}Q^{\prime \prime}$ and $P^{\prime \prime} _{v^{\prime}}$ would be a monophonic $(x,y),(u^{\prime },v^{\prime })$-path passing through $(u,v)$, a contradiction. We conclude that $\{(u,v),(u^{\prime},v^{\prime})\} $ is a maximal monophonic position set. 
\end{proof}

\begin{lemma}\label{lem:lowermp}
    If $uu^{\prime }$ is an edge of $G$ and $vv^{\prime }$ an edge of $H$, then $\{ (u,v),(u^{\prime },v^{\prime }) \} $ is a maximal monophonic position set of $G \cp H$.
\end{lemma}
\begin{proof}
    Let $uu^{\prime }$ and $vv^{\prime }$ be edges in $G$ and $H$ respectively. We claim that $M = \{ (u,v),(u^{\prime },v^{\prime })\} $ is a maximal monophonic position set. Trivially $(u,v^{\prime })$ and $(u^{\prime },v)$ cannot be added to $M$. Suppose that for some $y \notin \{ v,v^{\prime }\} $ a vertex $(u,y)$ can be added to $M$ to form a monophonic position set $M^{\prime }$. Then by Lemma~\ref{lem:projection is in mp} $\{ v,v^{\prime },y\} $ is a monophonic position set of $H$. Hence there is an induced $y,v$-path $Q$ that does not pass through $v^{\prime }$. The path $_uQ$ followed by $(u,v) \sim (u^{\prime },v) \sim (u^{\prime },v^{\prime })$ would then be an induced path through all three vertices of $M^{\prime } $. The case of adding a vertex from the $u^{\prime } $-, $v$- or $v^{\prime } $-layers is analogous. 

    Finally, consider a set $M^{\prime } = M \cup \{ (x,y) \} $, where $x \notin \{ u,u^{\prime } \} $ and $y \notin \{ v,v^{\prime } \} $. If $M^{\prime }$ is in monophonic position, then by Lemma~\ref{lem:projection is in mp} the sets $\{ u,u^{\prime },x\} $ and $\{ v,v^{\prime },y\} $ are in monophonic position in $G$ and $H$ respectively. Let $P$ be an induced $x,u^{\prime }$-path in $G$ that does not contain $u$ and $Q$ be an induced $y,v^{\prime }$-path in $H$ that does not contain $v$. Then the concatenation of $_xQ$ and $P_{v^{\prime }}$ followed by the path $(u^{\prime },v^{\prime }) \sim (u^{\prime },v) \sim (u,v)$ would be an induced path through all three vertices of $M^{\prime }$. 
\end{proof}

The {\em lower monophonic position number} $\mono ^-(G)$ of a graph $G$ is the cardinality of a smallest maximal monophonic position set of $G$~\cite{DiKlKrTu-2025}. Lemma~\ref{lem:lowermp} allows us to immediately deduce the lower monophonic position number of any Cartesian product.

\begin{corollary}
If $G,H$ are graphs with order at least two, then $\mono ^-(G \cp H) = 2$.
\end{corollary}

We now show that Lemmas~\ref{lemma:interval} and~\ref{lem:lowermp} yield an upper bound on the number of vertices in a varied monophonic position set of $G \cp H$.

\begin{corollary}\label{cor:varied sets are small}
If $S$ is a varied monophonic position set of $G \cp H$, then $|S| \leq 2$.
\end{corollary}

\begin{proof}
Assume that $S$ is a varied monophonic position set of $G \cp H$ and suppose for a contradiction that $|S| \geq 3 $. Suppose that $\pi _G(S)$ is an independent set; then by Lemma~\ref{lemma:interval} the set $\pi _H(S)$ would induce a clique, a contradiction. Thus, since $\pi _G(S)$ induces neither a clique nor an independent set in $G$, we can choose three vertices $(u_1,v_1)$, $(u_2,v_2)$ and $(u_3,v_3)$ in $S$ such that $\{ u_1,u_2,u_3\} $ induces neither a triangle nor an independent set. Hence by Lemma~\ref{lem:projection is in mp} we can assume that $u_1 \sim u_2$, $u_1 \not \sim u_3$ and $u_2 \not \sim u_3$. Hence by Lemmas~\ref{lemma:interval} and~\ref{lem:lowermp} we must have $v_1 \sim v_3$, $v_2 \sim v_3$ and $v_1 \not \sim v_2$; however, this implies that $v_1,v_3,v_2$ is an induced path, contradicting Lemma~\ref{lem:projection is in mp}.  

\end{proof}

\subsection{Cliquey and layered sets}

We now restrict our attention to layered and cliquey monophonic position sets. Without loss of generality, we will assume that layered sets lie in a $H$-layer $\{ u\} \times V(H)$ and that if $S$ is cliquey, then $\pi _G(S)$ is a clique of order at least two. We will use the following labelling convention for layered and cliquey monophonic position sets $S$ of $G \cp H$: with each vertex $u_i \in \pi _G(S)$ we associate the set $S_i^{\prime } = \pi _H(S \cap (\{ u_i\} \times V(H))$ (hence if $S$ is layered there is just one such set). Recall that by Lemma~\ref{projection} the sets $S_i^{\prime }$ are pairwise disjoint, so that these sets partition $\pi _H(S)$. The sets $S_i^{\prime}$ are schematically represented in Fig.~\ref{fig:sets-S_i'}.

\begin{figure}[ht!]
\begin{center}
\begin{tikzpicture}[scale=1.3,style=thick,x=1cm,y=1cm]
\def\vr{2pt}

\coordinate(x1) at (1.5,0.5);
\coordinate(x2) at (2.2,1);
\coordinate(x3) at (2.2,1.5);
\coordinate(x4) at (2.2,2);
\coordinate(x5) at (2.9,2.5);
\coordinate(x6) at (3.6,3.0);
\coordinate(x7) at (3.6,3.5);
\coordinate(y1) at (1.5,-0.5);
\coordinate(y2) at (2.2,-0.5);
\coordinate(y3) at (2.9,-0.5);
\coordinate(y4) at (3.6,-0.5);
\coordinate(z1) at (0,0.5);
\coordinate(z2) at (0,1);
\coordinate(z3) at (0,1.5);
\coordinate(z4) at (0,2);
\coordinate(z5) at (0,2.5);
\coordinate(z6) at (0,3.0);
\coordinate(z7) at (0,3.5);

\draw (-1.3,0.1) -- (6.3,0.1);
\draw (0.5,-1.2) -- (0.6,5.3);
\draw (y1) -- (y4);
\draw (y1) .. controls (1.8,-0.2) and (2.0,-0.2) .. (y3);
\draw (y2) .. controls (2.5,-0.2) and (3.4,-0.2) .. (y4);
\draw (y1) .. controls (1.6,0) and (3.5,0) .. (y4);
\draw[rounded corners](-1.3,0.3)--(-1.3,5)--(0.3,5)--(0.3,0.3)--cycle;
\draw[rounded corners](1.0,-1.2)--(1.0,-0.1)--(6,0.-0.1)--(6,-1.2)--cycle;
\draw[rounded corners](1.0,0.3)--(1.0,5)--(6,5)--(6,0.3)--cycle;
\draw[rounded corners](-0.2,0.35)--(-0.2,0.75)--(0.2,0.75)--(0.2,0.35)--cycle;
\draw[rounded corners](-0.2,0.85)--(-0.2,2.15)--(0.2,2.15)--(0.2,0.85)--cycle;
\draw[rounded corners](-0.2,2.35)--(-0.2,2.65)--(0.2,2.65)--(0.2,2.35)--cycle;
\draw[rounded corners](-0.2,2.85)--(-0.2,3.65)--(0.2,3.65)--(0.2,2.85)--cycle;
\foreach \i in {1,2,3,4,5,6,7} 
{
\draw(x\i)[fill=white] circle(\vr);
\draw(z\i)[fill=white] circle(\vr);
}
\foreach \i in {1,2,3,4} 
{
\draw(y\i)[fill=white] circle(\vr);
}
\node at (1.5,-1) {$u_1$};
\node at (2.2,-1) {$u_2$};
\node at (2.9,-1) {$u_3$};
\node at (3.6,-1) {$u_4$};
\node at (5.5,-0.5) {$G$};
\node at (-1,4.5) {$H$};
\node at (5,4.5) {$G\cp H$};
\node at (-0.5,0.5) {$S_1'$};
\node at (-0.5,1.5) {$S_2'$};
\node at (-0.5,2.5) {$S_3'$};
\node at (-0.5,3.25) {$S_4'$};
\end{tikzpicture}
\caption{Sets $S_i^{\prime}$}
\label{fig:sets-S_i'}
\end{center}
\end{figure}

Since we can find monophonic position sets in $G \cp H$ of order $\max \{ \omega (G),\omega (H)\} $, we also tacitly assume that the monophonic sets under discussion have order strictly greater than the lower bound in Observation~\ref{obs:lower bound}.  

\begin{lemma}
\label{lem:S' is independent} Let $S$ be either a cliquey or a layered monophonic position set of $G\cp H$ with $|S| > \max \{ \omega (G),\omega (H)\} $. Then:
\begin{enumerate}
    \item[$($a$)$] If $\{ u\} \times S^{\prime } \subseteq S$, where $S^{\prime } \subseteq V(H)$, then $S^{\prime }$ induces an independent set in $H$. 
    \item[$($b$)$] If $S^{\prime \prime } \times \{ v\} \subseteq S$, where $S^{\prime\prime } \subseteq V(G)$, then $S^{\prime\prime }$ induces an independent set in $G$.
   \end{enumerate}
\end{lemma}

\begin{proof} By the commutativity of $G\cp
H$, it suffices to prove the first part.
Suppose that $\{ u\} \times S^{\prime } \subseteq S$. We show firstly that $S^{\prime }$ is either a clique or an independent set. If $|S^{\prime }| \leq 2$, then there is nothing to prove, so we assume that $|S^{\prime }| \geq 3$. By Lemmas~\ref{lemma:2.1} and~\ref{lem:projection is in mp}, $S^{\prime }$ is an independent union of cliques, so if $S^{\prime }$ is not a clique or an independent set, then there are vertices $v_1,v_2,v_3 \in S^{\prime }$ such that $v_1 \sim v_2$, $v_1 \not \sim v_3$ and $v_2 \not \sim v_3$ in $H$. Let $u^{\prime }$ be any neighbour of $u$ in $G$ and $Q$ be a monophonic $v_2,v_3$-path in $H$. Since $\pi_H(S)$ is a monophonic position set of $H$, it is clear that the vertex $v_1$ does not belong to the path $Q$. Then the path formed by concatenating the path $(u,v_1),(u,v_2),(u^{\prime },v_2)$ with the path $_{u^{\prime }}Q$ and then the edge $(u^{\prime },v_3) \sim (u,v_3)$ in that order would be $\{ u\} \times S^{\prime }$-bad. Thus $S^{\prime }$ is either a clique or an independent set.

    If $S$ is layered, and $S^{\prime }$ is a clique, we would immediately have $|S| \leq \omega (H)$, so we assume that $S$ is cliquey. Suppose that there are $v_1,v_1^{\prime } \in S^{\prime }$ with $v_1 \sim v_1^{\prime }$ in $H$ and let $u_2 \in \pi _G(S) \setminus \{ u\} $, where $(u_2,v_2) \in S$. Let $Q$ be a monophonic $v_1^{\prime },v_2$-path. Then the path $(u,v_1),(u,v_1^{\prime }),(u_2,v_1^{\prime })$ followed by $_{u_2}Q$ is a bad path, so there is no such edge in $S^{\prime }$. 
\end{proof}

\begin{lemma}\label{cliquey}
    Let $S$ be a cliquey monophonic position set of $G \cp H$ with $|S| > \max \{ \omega (G),\omega (H)\} $, where $\pi _G(S)$ is a clique. Then $\pi _H(S)$ is an independent set.
\end{lemma}
\begin{proof}
Lemma~\ref{lem:S' is independent} shows that each set $S_i^{\prime }$ is an independent set in $H$. It further follows from Lemma~\ref{lem:lowermp} that there cannot be an edge in $H$ between sets $S_i^{\prime}$ and $S_j^{\prime}$ for $i \neq j$.
\end{proof}

\begin{lemma}
\label{lem:simplicial vertices}
If $S$ is a layered or cliquey monophonic position set of $G \cp H$ (where $\pi _G(S)$ is a clique) and $|S| > \max \{ \omega (G),\omega (H)\} $, then every vertex from $\pi _G(S)$ is simplicial.
\end{lemma}

\begin{proof}
Suppose that $u_i\in \pi _G(S)$ is not simplicial, with $w_1,w_2 \in N_G(u_i)$ and $w_1 \not \sim w_2$ in $G$. Firstly assume that $|S_i^{\prime }| \geq 3$, with $\{ v_1,v_2,v_3\} \subseteq S_i^{\prime }$. Let $Q$ and $Q^{\prime }$ be monophonic $v_1,v_2$- and $v_2,v_3$-paths respectively. By Lemmas~\ref{lem:projection is in mp} and~\ref{lem:S' is independent}, the path formed by concatenating the edge $(u_i,v_1) \sim (w_1,v_1)$, the path $_{w_1}Q$, the path $(w_1,v_2)$, $(u_i,v_2)$, $(w_2,v_2)$, the path $_{w_2}Q^{\prime }$ and finally the edge $(w_2,v_3) \sim (u_i,v_3)$ in that order would be bad.

Now suppose that $|S_i^{\prime }| = 2$. If $S$ is layered, then we could not have $|S| > \max \{ \omega (G),\omega (H)\} $, so $S$ must be cliquey and there is a $u_j \in \pi _G(S) \setminus \{ u_i\} $. Let $v_1,v_2 \in S_i^{\prime }$, $v_3 \in S_j^{\prime }$, $Q$ and $Q^{\prime }$ be monophonic $v_1,v_2$- and $v_2,v_3$-paths respectively and $P$ be a shortest path from $\{ w_1,w_2\} $ to $u_j$, which we assume without loss of generality to be a $w_2,u_j$-path (notice that $Q$ has length at most two and that $w_2 = u_j$ is possible). Then the following concatenated path is bad: the edge $(u_i,v_1) \sim (w_1,v_1)$, the path $_{w_1}Q$, the path $(w_1,v_2)$, $(u_i,v_2)$, $(w_2,v_2)$, the path $_{w_2}Q^{\prime}$, and finally the path $P_{v_3}$. Thus $|S_i^{\prime }| = 1$.

Finally suppose that $|S_i^{\prime }| = 1$. As $|S| > \max \{ \omega (G),\omega (H)\} $, $S$ is cliquey and there is a $u_j \in \pi _G(S) \setminus \{ u_i\} $ with $|S_j^{\prime }| \geq 2$, say $S_i^{\prime } = \{ v_1\} $ and $\{ v_2,v_3\} \subseteq S_j^{\prime }$. Since $|S_j^{\prime }| \geq 2$, the arguments in the previous case show that $u_j$ is simplicial. Consequently, $u_j$ is not adjacent to at least one of $w_1$ and $w_2$, say $w_1$. Let $Q$ and $Q^{\prime }$ be monophonic $v_1,v_2$- and $v_2,v_3$-paths respectively and $P$ be a monophonic $w_1,u_j$-path. The path formed as the concatenation of the edge $(u_i,v_1) \sim (w_1,v_1)$ and the paths $_{w_1}Q$, $P_{v_2}$, and $_{u_j}Q^{\prime }$ in that order would be bad.  
\end{proof}

A vertex subset $S$ of a graph that is simultaneously an independent set and in monophonic position is called an \emph{independent monophonic position set}. The largest order of an independent monophonic position set is the \emph{independent monophonic position number} of $G$, denoted by $\mono_i(G)$. With this notation in hand, we prove the following tight bounds. Recall that for any graph $\sigma (G) = 1$ if $G$ contains a simplicial vertex and $\sigma (G) = 0$ otherwise.
    
\begin{theorem} 
\label{thm:general bounds}
If $G$ and $H$ are connected graphs, then $$\max\{\omega(G), \omega(H)\}\leq \mono (G \cp H) \leq \max \{ \omega (G), \omega (H), \sigma (G) \mono_i(H), \sigma (H) \mono_i(G)\}.$$ 
    Furthermore, if neither $G$ nor $H$ has simplicial vertices, then $$\mono (G \cp H) = \max \{ \omega (G), \omega (H)\}\,.$$    
\end{theorem}

\begin{proof}
Suppose that $G\cp H$ contains a monophonic position set $S$ with $|S|> \max\{\omega(G), \omega(H)\}$. Then it follows from Corollary~\ref{cor:varied sets are small} that $S$ is either cliquey or layered. As before, we assume that $\pi _G(S)$ is a clique. Then by Lemma~\ref{lem:simplicial vertices} it is necessary that every vertex of $\pi _G(S)$ is simplicial, so that $\sigma (G) = 1$. Lemmas~\ref{lem:projection is in mp} and~\ref{cliquey} show that $\pi_H(S)$ is an independent monophonic position set of $H$ and hence $\mono(G\cp H)=|S|=|\pi_H(S)|\leq \mono_i(H)$. The case that $\pi _H(S)$ is a clique is similar and would yield an upper bound of $\sigma (H)\mono_i(G)$. 
\end{proof} 

We conclude this subsection by a lower bound for graphs which contain leaves. To do so, for a vertex of any graph $G$ we will write $\ell (u)$ for the number of neighbours of $u$ that are leaves, and set $\Delta _1(G) = \max \{ \ell (u):u \in V(G)\} $. 

\begin{proposition}\label{prop:lowerbound}
    If $G$ and $H$ are graphs with order at least three that both contain leaves, then \[ \mono (G \cp H) \geq \max \{ \Delta _1(G),\Delta _1(H)\} .\]
\end{proposition}

\begin{proof}
    Let $u$ be a leaf of $G$ and $L$ a set of leaves adjacent to some vertex $v$ of $H$. We show that the set $\{ u\} \times L$ is in monophonic position in $G\cp H$. If $|L| \leq 2$, then this is trivial. Otherwise, suppose that $P$ is a monophonic $(u,x_1),(u,x_3)$-path in $G \cp H$ passing through $(u,x_2)$ for some $x_1,x_2,x_3 \in L$. The vertex $(u,x_2)$ has degree two in $G \cp H$ and so one of its neighbours is $(u,v)$, which is also adjacent to $(u,x_1)$ and $(u,x_2)$, a contradiction. 
\end{proof}

As an example, consider $K_{1,n}\cp K_{1,m}$, $n\ge m\ge 2$, $n\ge 4$. Let $V(K_{1,k})=\{0,1,\ldots, k\}$, where $0$ is the central vertex. Then the set $T= [n]\times \{1\}$ is a maximum monophonic position set of $K_{1,n}\cp K_{1,m}$ and hence $\mono(K_{1,n}\cp K_{1,m}) = n = \mono_i(K_{1,n})$. This demonstrates that the upper bound in Theorem~\ref{thm:general bounds} and the lower bound in Proposition~\ref{prop:lowerbound} are sharp. In addition, it follows from~\cite[Theorem 2.1]{TianXuChao-2023} that $\gp(K_{1,n}\cp K_{1,m}) = m+n$, hence this is another example that shows that the $\mono $-number and the $\gp $-number on product graphs can be arbitrarily far apart.

\subsection{Triangle-free graphs}

In this subsection we employ the structural properties established in the earlier lemmas about the monophonic position sets in the Cartesian product of graphs to give a tighter bound on the monophonic position number of Cartesian products of triangle-free graphs. Observe that any simplicial vertex of a triangle-free graph must be a leaf.

\begin{theorem}
\label{thm:bipartite}
If $G$ and $H$ are connected triangle-free graphs of order at least $3$, then 
$$\mono(G\cp H)\leq\max\{2, \sigma(G)\Delta(H), \sigma(H)\Delta(G)\}\,.$$ 
\end{theorem}

\begin{proof} 
Let $S$ be a largest monophonic position set of $G\cp H$. If $|S|\geq 3$, then $S$ is either cliquey or layered by Corollary \ref{cor:varied sets are small}. Suppose that $S$ is cliquey and $\pi_G(S)=\{u_1, u_2\}$ induces a clique, that is, $u_1\sim u_2$ in $G$. However, as the order of $G$ is at least three, the vertices $u_1$ and $u_2$ cannot both be leaves, contradicting Lemma~\ref{lem:simplicial vertices}. 

Hence $S$ must be layered. Again by Lemma~\ref{lem:simplicial vertices}, we may assume that $\pi_G(S)=\{u\}$, where $u$ is a leaf of $G$. Let $u^{\prime }$ be any vertex of $G$ with $d_G(u,u^{\prime }) = 2$ and $P$ be any shortest $u,u^{\prime }$ path. As in the proof of Theorem~\ref{thm:general bounds}, we conclude that $\pi_H(S)$ is an independent monophonic position set of $H$. 

We first show that all vertices of $\pi _H(S)$ are at distance two from each other. Let $k = \max \{ d_H(v,v^{\prime }):v,v^{\prime } \in \pi _H(S)$. Assume for a contradiction that $k \geq 3$ and let $v_1,v_2$ be a pair of vertices of $\pi _H(S)$ with $d_H(v_1,v_2) = k$. Let $v_3$ be any third vertex of $\pi _H(S)$ ($v_3$ exists, since $|\pi _H(S)| = |S| \geq 3$). Amongst all shortest $v_1,v_2$-paths in $H$, let $Q$ be a path that minimises $d_H(Q,v_3)$, where $d_H(Q,v_3) = \min \{ d_H(x,v_3):x \in V(Q)\} $. Let $Q$ be the path $v_1 = x_0,x_1,\dots x_{k-1},x_k = v_2$. 

Let $x_i$ be the vertex with the smallest value of $i$ on $Q$ such that $d_H(x_i,v_3) = d_H(Q,v_3)$. Let $Q^+$ be the monophonic $v_3,v_1$-path in $H$ formed from the concatenation of a shortest $v_3,x_i$-path $Q^*$ and the $x_i,v_1$-section of $Q$. If $i = k$, then $Q^+$ would be a monophonic path containing all three of $v_1,v_2,v_3$, so we have $0 \leq i \leq k-1$.   

Suppose that $0 \leq i \leq k-2$. Consider the concatenation of the four paths $_uQ^+$, $P_{v_1}$, $_{u^{\prime }}Q$ and $\widetilde{P}_{v_2}$ in that order. As $v_2$ does not lie on $Q^+$, $d_G(u,u^{\prime }) = 2$ and $v_1 \not \sim v_2$ in $H$, this concatenated path would be a monophonic path containing $(u,v_1)$, $(u,v_2)$ and $(u,v_3)$ unless $v_2$ has a neighbour $y$ on $Q^*$. By our assumption that $i \leq k-2$, $y \neq x_i$. In fact, $y$ must be the neighbour of $x_i$ on $Q^*$, for otherwise $Q^*$ would not be a shortest path from $v_3$ to $Q$. As $Q$ is a shortest $v_1,v_2$-path, it follows that $i = k-2$ and the path formed by following the $v_1,v_{k-2}$-section of $Q$ followed by the path $v_{k-2},y,v_2$ would be a shortest $v_1,v_2$-path closer to $v_3$ than $Q$, contradicting the definition of $Q$. Hence $i = k-1$.   

Now the vertex $v_2$ has no neighbours on $Q^*$ other than $x_{k-1}$, for otherwise $H$ would either contain a triangle or else $Q^*$ would not be a shortest path from $v_3$ to $Q$. Therefore the $v_3,v_2$-path $Q^{++}$ formed by appending the edge $x_{k-1}v_2$ to $Q^*$ is monophonic. By definition of $Q^*$ and the condition that $d_H(v_1,v_2) = k \geq 3$ there is no edge from $v_1$ to $Q^{++}$ and it follows that the concatenation of the four paths $_uQ^{++}$, $P_{v_2}$, $_{u^{\prime }}\widetilde{Q}$ and $\widetilde{P}_{v_1}$ is a monophonic path containing each of $(u,v_1)$, $(u,v_2)$ and $(u,v_3)$. Therefore our assumption that $k \geq 3$ is false and all pairs of vertices from $\pi _H(S)$ are at distance two.

Suppose now that $\mono (G \cp H) > \Delta (H)$. Fix two vertices $v_1,v_2 \in \pi _H(S)$ and let $Q$ be a shortest path $v_1,x,v_2$ between them. As $|\pi _H(S)| > \Delta (H)$, there must be a vertex $v_3$ of $\pi _H(S)$ that is not adjacent to $x$. Let $Q'$ be a shortest path $v_2,y,v_3$. As there is no edge from $v_3$ to $\{ v_1,x,v_2\} $, it follows that the concatenation of $_uQ$, $P_{v_2}$, $_{u^{\prime }}Q'$ and $\widetilde{P}_{v_3}$ in that order would be a monophonic path containing $(u,v_1)$, $(u,v_2)$ and $(u,v_3)$, a contradiction. Therefore $\mono (G \cp H) \leq \Delta (H)$. The claimed inequality follows when we allow the roles of $G$ and $H$ to be reversed.
\end{proof}

Again the example of the Cartesian product of two stars at the end of the previous subsection shows that the bound of Theorem~\ref{thm:bipartite} is tight.

\section{Lexicographic products}\label{sec:Lexicographic products}

The variety of general position sets was introduced in~\cite{Tian-2025} and has already been investigated on lexicographic products in~\cite{Dokyeesun-2024b}. In this section, we determine the monophonic position number of arbitrary lexicographic products. Before stating the main result of this section, we recall the distance function of lexicographic products (see~\cite[Proposition 5.12]{hik-2011}) and state two lemmas.  

\begin{proposition}
\label{prop:Products:LexDist1}
If $(g,h)$ and $(g^{\prime},h^{\prime})$ are distinct vertices of $G\circ H$, then
$$d_{G\circ H}\left((g,h),(g^{\prime},h^{\prime})\right)=
\left\{\begin{array}{ll}
d_G(g,g^{\prime}); & \mbox{$g\ne g^{\prime}$}\,,\\
d_H(h,h^{\prime}); & g=g^{\prime}, \deg_G(g)=0\,,\\
\min\{d_H(h,h^{\prime}), 2\}; & g=g^{\prime}, \deg_G(g)\ne 0\,.\\
\end{array}\right.$$
\end{proposition}

\begin{lemma}
\label{lemma:6.4}
If $S$ is a monophonic position set of $G \circ H$, then $\pi_{G}(S)$ is a monophonic position set of $G$.
\end{lemma}

\begin{proof}
Suppose for a contradiction that $\pi_G (S)$ is not a monophonic position set in $G$. Then there exist vertices $u_1,u_2,u_3 \in \pi_{G}(S)$ and a monophonic  $u_1,u_3$-path $P$ in $G$ which passes through $u_2$. As $u_1,u_2,u_3 \in \pi_{G}(S)$, there exist vertices $v_1,v_2,v_3\in V(H)$ such that $(u_1,v_1), (u_2,v_2), (u_3,v_3)\in S$. But now it is straightforward to lift the path $P$ to a monophonic $(u_1,v_1),(u_3,v_3)$-path in $G\circ H$ which passes through $(u_2,v_2)$, a contradiction. 
\end{proof}

Let $u\in V(G)$ and and let $P$ be a monophonic path in $H$. Then the isomorphic copy of $P$ in the layer $^uH$ of $G\circ H$ is a monophonic path of $G\circ H$. This fact implies the second lemma of this section.   

\begin{lemma}
\label{lem:proj_in_H}
If $S$ is a monophonic position set of $G\circ H$, then for any $u\in \pi_{G}(S)$, $\pi_{H}({}^uH\cap S)$ is a monophonic position set of $H$.
\end{lemma}

Note that $\pi_{H}(S)$ need not be a monophonic position set of $H$. As an example consider $P_2\circ P_3$ (See Fig.~\ref{fig:6.2}). Let $V(P_2)=\{u_1,u_2\}$, $V(P_3)=\{v_1,v_2,v_3\}$ and let $S=\{(u_1,v_1),(u_1,v_2),(u_2,v_2),(u_2,v_3)\}$. As $S$ is a clique, it is a monophonic position set of $P_2\circ P_3$. But $\pi_{P_3}(S)=\{v_1,v_2,v_3\}$ is not a monophonic position set in $P_3$.\\ \\

\begin{figure}[ht!]
	\centering
	\begin{tikzpicture}[x=0.2mm,y=-0.2mm,inner sep=0.2mm,scale=1,thick,vertex/.style={circle,draw,minimum size=8}]
		\node at (0,50) [vertex,fill=black] (x1) {};
            \node at (100,50) [vertex,fill=black] (x2) {};
            \node at (200,50) [vertex] (x3) {};
  
            \node at (0,-50) [vertex] (y1) {};
            \node at (100,-50) [vertex,fill=black] (y2) {};
            \node at (200,-50) [vertex,fill=black] (y3) {};
        
		\path
		(x1) edge (x2)
		(x3) edge (x2)

            (y1) edge (y2)
		(y3) edge (y2)

            (x1) edge (y1)
            (x1) edge (y2)
            (x1) edge (y3)

            (x2) edge (y1)
            (x2) edge (y2)
            (x2) edge (y3)

            (x3) edge (y1)
            (x3) edge (y2)
            (x3) edge (y3)
		
		;
	\end{tikzpicture}
	\caption{A monophonic position set of $G\circ H$ need not project to a monophonic position set in $H$}
	\label{fig:6.2}
\end{figure}
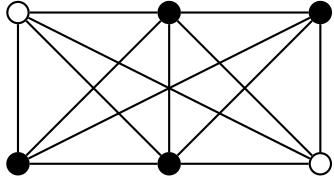   

In view of Lemma~\ref{lemma:2.1}, in the rest of this section, for a monophonic position set $M$ of $G$ we denote the components of $G[M]$ by: $A_1,A_2,\ldots,A_k,B_1,\dots ,B_r $, where $|A_i|\geq 2$ for each $i\in[k]$ and $|B_j|=1$ for each $j\in [r]$. Also we fix $n_M = \sum _{i=1}^k |A_{i}|$ and write $r_M = r$ to emphasise that $r$ is a function of $M$. Then for any monophonic position set $M$ of $G$ we have $|M| = n_M+r_M$. Now we are ready for our main result of this section.

\begin{theorem}
\label{thm:6.5}
Let $G$ be a connected graph of order at least $2$ and let $\mathcal{M}$ be the collection of all monophonic position sets of $G$. Then $$ \mono(G\circ H)=\max_{M\in \mathcal{M}}\{n_M\cdot\omega(H)+ r_M\cdot \mono(H)\}\,.$$
\end{theorem}

\begin{proof}  
Let $S$ be a monophonic position set of $G\circ H$. Set $M = \pi _G(S)$, which is a monophonic position set of $G$ by Lemma~\ref{lemma:6.4}. To show that $|S| \leq \max_{M\in \mathcal{M}}\{n_M\cdot\omega(H)+ r_M\cdot \mono(H)\} $ it suffices to show that if $u \in M$ lies in a component $A_i$ of $G[M]$ of order greater than one, then $\pi _H(^uH \cap S)$ is a clique. Suppose for a contradiction that there are $v_1,v_2 \in {}^uH \cap S$ such that $v_1 \not \sim v_2$ in $H$. Let $u^{\prime } \in A_i \setminus \{ u\}$ and $v^{\prime }$ be a vertex of $H$ such that $(u^{\prime },v^{\prime }) \in S$. Then $(u,v_1),(u^{\prime },v^{\prime }),(u,v_2)$ would be a bad monophonic path and the claimed bound holds.

For the converse, we show that any subset $S$ of $V(G \circ H)$ with the properties that \begin{itemize}
    \item $M = \pi _G(S)$ is in monophonic position in $G$,
    \item $\pi _H(^uH \cap S)$ is a clique of $H$ if $u$ belongs to a component of $G[M]$ of order greater than one, and
    \item $\pi _H(^uH \cap S)$ is a monophonic position set of $H$ if $u$ is an isolated vertex of $G[M]$,
\end{itemize}  
then $S$ is in monophonic position in $G \circ H$. Let $S$ be such a subset, set $M = \pi _G(S)$ and define the sets $A_i$ and $B_j$ as above. Assume for a contradiction that there are vertices $(u_1,v_1)$, $(u_2,v_2)$ and $(u_3,v_3)$ of $S$ and a monophonic $(u_1,v_1),(u_3,v_3)$-path $P$ of $G \circ H$ passing through $(u_2,v_2)$.

If $u_1,u_2,u_3$ are distinct vertices of $G$, then $\pi _G(P)$ would be a monophonic path of $G$ containing $u_1,u_2,u_3$, contradicting our assumption that $M$ is in monophonic position in $G$. Suppose that $u_1 = u_3$. Then we must have $v_1 \not \sim v_3$ in $H$, for otherwise $P$ would have length just one; therefore $u_1 = u_3 \in B_j$ for some $j \in [r]$. We have $d_{G \circ H}((u_1,v_1),(u_3,v_3)) = 2$. Then either $u_2 = u_1$ and $v_1,v_2,v_3$ is a monophonic path in $H$, which contradicts Lemma~\ref{lem:proj_in_H}, or else $u_2 \sim u_1$, which is impossible since $u_1 \in B_j$. 

Hence we can assume that $u_1 = u_2 \neq u_3$. Then if $(u^{\prime },v^{\prime })$ is the first vertex on the $(u_2,v_2),(u_3,v_3)$-subpath of $P$ with $u^{\prime } \neq u_2$ there is a chord $(u_1,v_1) \sim (u^{\prime },v^{\prime })$, so that again $P$ is not monophonic. Thus $S$ is in monophonic position and the monophonic position number of $G \circ H$ is the maximum value of $n_M\omega (H)+r_M\mono (H)$ over all such sets.
\end{proof}

It follows from Lemma~\ref{lemma:2.1} that any monophonic position set of a triangle-free graph $G$ is either an independent set or consists of a pair of adjacent vertices, so if the order of the graph is $n \geq 3$ we have $\mono (G) = \mono _i (G)$. As $\omega (H) \leq \mono (H)$ this provides an exact value for $\mono (G \circ H)$ when $G$ is triangle-free.
\begin{corollary}
\label{corollary:3.6}
If $G$ be a connected triangle-free graph with order $n \geq 3$ and $H$ is a connected graph, then  $\mono (G\circ H) = \mono (G)\cdot \mono(H)$.
\end{corollary}

\begin{corollary}
\label{corollary:3.7}
If $H$ is a connected graph and $n\ge 2$, then \[ \mono (K_n\circ H)=\max \{ n\cdot \omega(H),\mono (H)\} .\]
\end{corollary}


\section*{Acknowledgements}

Sandi Klav\v zar was supported by the Slovenian Research and Innovation Agency ARIS (research core funding P1-0297 and projects  N1-0285, N1-0355, N1-0431). James Tuite gratefully acknowledges funding support from EPSRC grant EP/W522338/1. The authors also thank Grahame Erskine for helpful discussions of the paper.

\section*{Declaration of interests}

The authors declare that they have no known competing financial interests or personal relationships that could have appeared to influence the work reported in this paper.

\section*{Data availability}

Our manuscript has no associated data.


\end{document}